\documentclass[a4paper,english,10pt,oneside]{article}
\usepackage[english]{babel}
\usepackage{amsmath}
\usepackage{amssymb}
\usepackage{wasysym}

\usepackage{tikz}
\usepackage{pgfplots}

\usepackage{graphicx}
\usepackage{mathtools}
\usepackage{amsthm}
\usepackage{booktabs}
\usepackage{xspace}
\usepackage{hyperref}
\usepackage[font=small]{caption}

\usepackage{todonotes}

\graphicspath{{./pictures/}}

\definecolor{darkblue}{rgb}{0,0,.5}
\DeclareMathOperator{\rank}{rk} 
\DeclareMathOperator{\proj}{proj} 
\DeclareMathOperator{\ext}{ext} 
\DeclareMathOperator{\clext}{\overline{ext}} 
\DeclareMathOperator{\cone}{cone} 
\DeclareMathOperator{\shw}{shw} 
\DeclareMathOperator{\conv}{conv} 
\DeclareMathOperator{\clconv}{\overline{conv}} 
\DeclareMathOperator*{\argmax}{arg\,max} 
\DeclareMathOperator{\cl}{cl} 
\newcommand{\xlp}{\bar{x}} 

\newcommand{\st}{\,:\,}

\usepackage{xifthen}
\ifthenelse{\isundefined{\T}}{%
  \newcommand{\T}{\mathsf{T}}
  }{%
  \renewcommand{\T}{\mathsf{T}}
}

\theoremstyle{plain}
\newtheorem{theorem}{Theorem}
\newtheorem{proposition}[theorem]{Proposition}
\newtheorem{lemma}[theorem]{Lemma}
\newtheorem{corollary}[theorem]{Corollary}
\theoremstyle{definition}
\newtheorem{definition}[theorem]{Definition}
\newtheorem{remark}[theorem]{Remark}
\newtheorem{question}[theorem]{Question}
\newtheorem{observation}[theorem]{Observation}
\newtheorem{XxmpX}[]{Example}
\newenvironment{example}{\pushQED{\qed}\begin{XxmpX}}{\popQED\end{XxmpX}}

\usepackage{zibtitlepage}

\newcommand{\myorcidlink}[1]{\,\href{https://orcid.org/#1}{\raisebox{-0.45ex}{\includegraphics[width=1.8ex]{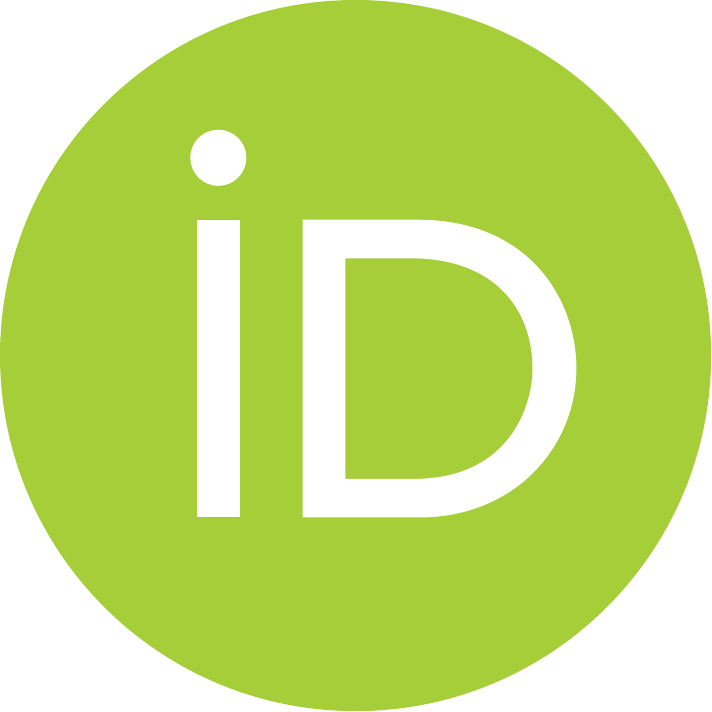}}}}

\ZTPAuthor{Felipe~Serrano\protect\myorcidlink{0000-0002-7892-3951}}
\ZTPTitle{\bf Visible points, the separation problem, and applications to MINLP}

\ZTPInfo{This work has been supported by the Research Campus MODAL
\emph{Mathematical Optimization and Data Analysis Laboratories} funded by the
Federal Ministry of Education and Research (BMBF Grant~05M14ZAM).\\
The author thanks the Schloss Dagstuhl -- Leibniz Center for Informatics for
hosting the Seminar 18081 "Designing and Implementing Algorithms for
Mixed-Integer Nonlinear Optimization" for providing the environment to develop
the ideas in this paper.}

\ZTPNumber{19-38}
\ZTPMonth{July}
\ZTPYear{2019}

\title{\bf Visible points, the separation problem, and applications to MINLP}
\author{
  Felipe Serrano\protect\myorcidlink{0000-0002-7892-3951}\thanks{Zuse Institute Berlin, Takustr.~7, 14195~Berlin,
  Germany, \texttt{serrano@zib.de}}
}

\begin{document}
\zibtitlepage
\maketitle

\begin{abstract}
  In this paper we introduce a technique to produce tighter cutting planes for
  mixed-integer non-linear programs.
  Usually, a cutting plane is generated to cut off a specific infeasible point.
  The underlying idea is to use the infeasible point to restrict the feasible
  region in order to obtain a tighter domain.
  To ensure validity, we require that every valid cut separating the infeasible
  point from the restricted feasible region is still valid for the original
  feasible region.
  We translate this requirement in terms of the separation problem and the
  reverse polar.
  In particular, if the reverse polar of the restricted feasible region is the
  same as the reverse polar of the feasible region, then any cut valid for the
  restricted feasible region that \emph{separates} the infeasible point, is
  valid for the feasible region.

  We show that the reverse polar of the \emph{visible points} of the feasible
  region from the infeasible point coincides with the reverse polar of the
  feasible region.
  In the special where the feasible region is described by a single non-convex
  constraint intersected with a convex set we provide a characterization of the
  visible points.
  Furthermore, when the non-convex constraint is quadratic the characterization
  is particularly simple.
  We also provide an extended formulation for a relaxation of the visible points
  when the non-convex constraint is a general polynomial.

  Finally, we give some conditions under which for a given set there is an
  inclusion-wise smallest set, in some predefined family of sets, whose reverse
  polars coincide.

  {\bf Keywords:} Separation problem, Visible points, Mixed-integer nonlinear
  programming, Reverse polar, Global optimization.
\end{abstract}

\section{Introduction}

The separation problem is a fundamental problem in
optimization~\cite{GroetschelLovaszSchrijver1993}.
Given a set $S \subseteq \mathbb{R}^n$ and a point $\xlp$, the separation
problem is
\begin{quote}
  Decide if $\xlp \in \clconv(S)$ or find an inequality $\alpha x \leq
  \beta$ that separates $\xlp$ from $\clconv(S)$.
\end{quote}

Algorithms to solve optimization problems, especially those based on solving
relaxations, such as branch and bound, need to deal with the separation problem.
Consider, for example, solving a mixed integer linear problem via branch and
bound~\cite[Section 9.2]{ConfortiCornuejolsZambelli2014}.
The solution to the linear relaxation plays the role of $\xlp$, while
a relaxation based on a subset of the constraints is used as $S$ for the
separation problem, see~\cite[Chapter 6]{ConfortiCornuejolsZambelli2014}.

The separation problem can be rephrased in terms of the \emph{reverse
polar}~\cite{Balas1998,Zaffaroni2008} of $S$ at $\xlp$, defined as
\[
  S^{\xlp} = \{ \alpha \in \mathbb{R}^n: \alpha^\T (x - \xlp) \geq 1, \forall x
  \in S\}.
\]
The elements of $S^{\xlp}$ are the normals of the hyperplanes that separate
$\xlp$ from $\clconv(S)$.
Hence, the separation problem can be stated equivalently as
\begin{quote}
  Decide if $S^{\xlp}$ is empty or find an element from it.
\end{quote}

The point of departure of the present work is the following observation.
\begin{observation} \label{intro:observation}
  If there is a set $V$ such that $(S \cap V)^{\xlp} = S^{\xlp}$, then, as far
  as the separation problem is concerned, the feasible region can be regarded as
  $S \cap V$ instead of $S$.\\
  %
\end{observation}
%
%
A set $V$ such that $V^{\xlp} = S^{\xlp}$ will be called a \emph{generator} of
$S^{\xlp}$.
Intuitively, if a set $V$ is such that $V \cap S$ generates $S^{\xlp}$, that is,
if we can ensure that a cut valid for $V \cap S$ that separates $\xlp$ is also
valid for $S$, then $V$ should at least contain the points of $S$ that are
``near'' $\xlp$.
To formalize the meaning of ``near'' we use the concept of \emph{visible
points}~\cite{DeutschHundalZikatanov2013} of $S$ from $\xlp$, which are the
points $x \in S$ for which the segment joining $x$ with $\xlp$ only intersects
$S$ at $x$, see Definition~\ref{def:visible}.
In other words, they are the points of $S$ that can be ``seen'' from $\xlp$.
In Proposition~\ref{prop:visible_reverse} we show that the visible points are
a generator of $S^{\xlp}$.

As a motivation, we present an application of our results in the context of
nonlinear programming, which is treated in more detail in Section~\ref{sec:minlp}.
\begin{example} \label{intro:ex1}
  Consider the separation problem of $\xlp = (0,0)$ from $S = \{ x \in B : g(x)
  \leq 0\}$ where
  \begin{align*}
    B &= [-\tfrac{1}{2},3] \times [-\tfrac{1}{2},3],\\
    g(x_1,x_2) &= -x_1^2 x_2 + 5 x_1 x_2^2 - x_2^2 - x_2 -2 x_1 +2,
  \end{align*}
  as depicted in Figure~\ref{fig:ex1}.
  A standard technique for solving the separation problem for $S$ and $\xlp$ is
  to construct a convex underestimator of $g$ over $B$~\cite[Sections 6.1.2 and
  7.5.1]{Vigerske2013}.
  The quality of a convex underestimator depends on the bounds of the
  variables and tighter bounds yield tighter underestimators.
  As we will see (Proposition~\ref{prop:visible_reverse} and
  Lemma~\ref{lemma:characterization_visibility}), $R^{\xlp} = S^{\xlp}$ where
  \[
    R = \{ x \in B : g(x) = 0, \langle \nabla g(x), x \rangle \leq 0 \}.
  \]
  It is possible to show that $R \subseteq V$, where $V = [-\tfrac{1}{2},
  \tfrac{17}{10}] \times [-\tfrac{6}{25},\tfrac{3}{2}]$.
  Hence, by Corollary~\ref{cor:minlp:visibles_work}, $(V \cap S)^{\xlp}
  = S^{\xlp}$.
  This means that we can solve the separation problem over $\{ x \in V :
  g(x) \leq 0\}$ instead of $S$.
  Therefore, if we were to compute an underestimator of $g$, it could be
  computed over $V \subsetneq B$.
  \begin{figure}
    \centering
    \includegraphics[width=.4\linewidth]{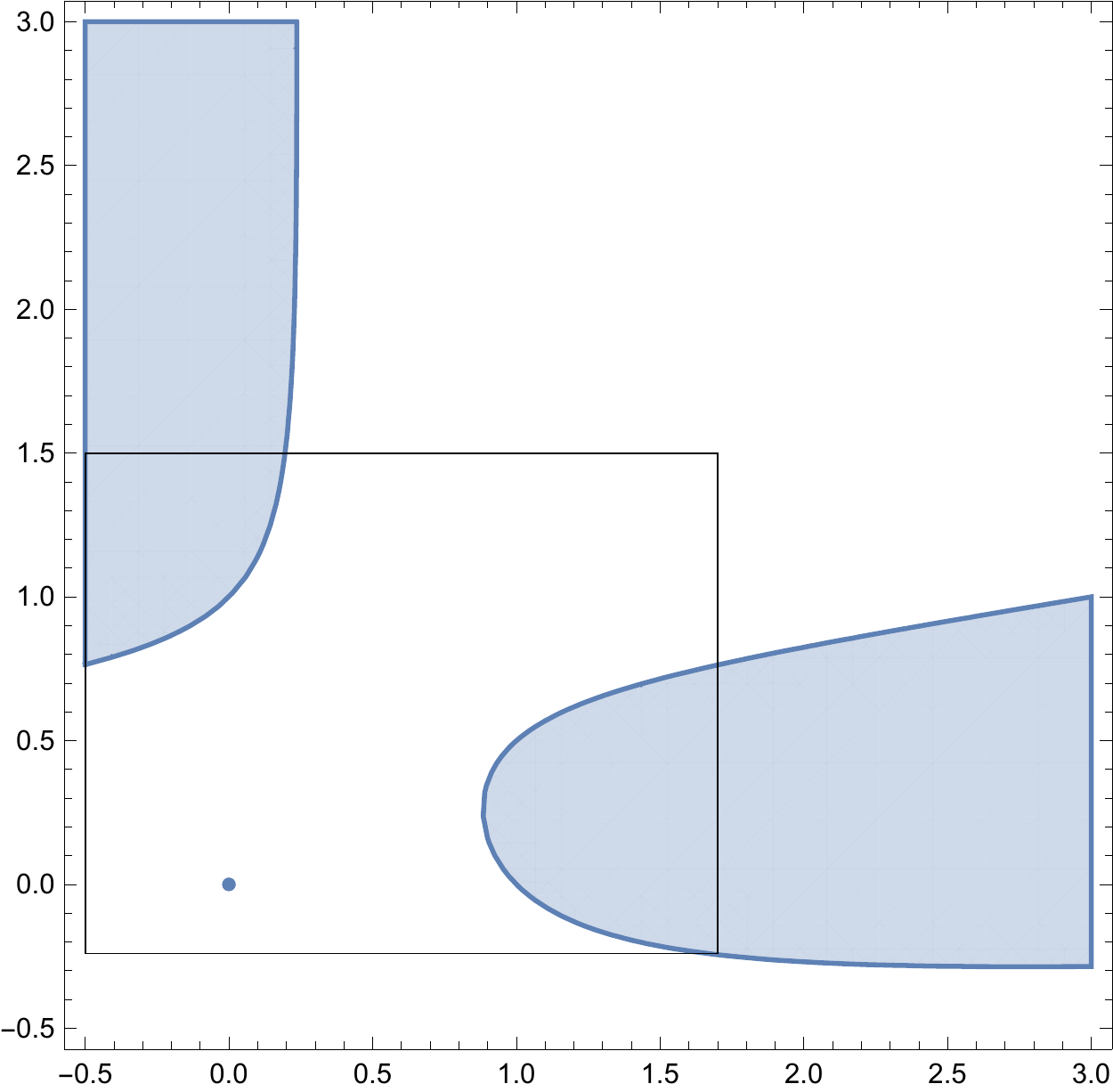}
    \caption{The feasible region $g(x) \leq 0$ and $\xlp = (0,0)$
      together with the box $V$.
    } \label{fig:ex1}
  \end{figure}
\end{example}

Methods for obtaining tighter bounds for mixed integer nonlinear
programming (MINLP) are of paramount importance.
Indeed, not only bound tightening procedures enhance the performance of MINLP
solvers, but also many algorithms for solving MINLPs require that all variables
are bounded~\cite{HamedMcCormick1993}.
We refer to the recent survey~\cite{PuranikSahinidis2017} for more information
on bound tightening procedures and its impact on MINLP solvers, and
to~\cite{BoukouvalaMisenerFloudas2016} for the practical importance of MINLP.

However, the technique that we introduce in this paper is \emph{not} a bound
tightening technique in the classic sense, i.e., the tighter bounds that might
be learned from $V$ are not valid for the original problem, but only for the
separation problem at hand.

We would like to point out that in~\cite{VenkatachalamNtaimo2016} a similar
idea --- to modify the separation problem --- is used in the context of
stochastic mixed integer programming.
The objective of the authors of~\cite{VenkatachalamNtaimo2016} is to speed-up
the solution of the separation problem.
In contrast, our objective is to produce tighter cutting planes for MINLP.

\paragraph{Contributions}
We show that for every closed set $S$, there exists an inclusion-wise smallest
closed convex set that generates $S^{\xlp}$ (Theorem~\ref{thm:smallest_convex}).
When $S$ is compact, there is an inclusion-wise smallest closed set that
generates $S^{\xlp}$ (Theorem~\ref{thm:smallest_closed}).
Furthermore, under some mild assumptions on $S$, we show that there is an
inclusion-wise smallest closed convex set $C$ such that $C \cap S$ generates
$S^{\xlp}$ (Theorem~\ref{thm:smallest_inter}).
We also show the existence of a generator, $V_S(\xlp)$, of $S^{\xlp}$ which is
more suitable for computations.

We apply our results to MINLP and give an explicit description of $V_S(\xlp)$
when $S = \{ x \in C : g(x) \leq 0\}$, where $C$ is a closed convex set
containing $\xlp$, and $g$ is continuous (Section~\ref{subsec:characterization_minlp}).
For the important case of quadratic constraints, i.e., when $g$ is a quadratic
function, we show that $V_S(\xlp)$ has a particularly simple expression
(Theorem~\ref{thm:qvis}).\\
For the case when $g$ is a general polynomial, we provide an extended
formulation for a relaxation of $V_S(\xlp)$ based on the theory of non-negative
univariate polynomials (Theorem~\ref{thm:poly_vis}).

\paragraph{Definitions and notation}
Given a set $S$, $\conv S, \cl S, \clconv S, \ext S$ represent the convex hull,
the closure, the closure of the convex hull and the extreme points of $S$,
respectively.
The \emph{extreme points} of a, not necessarily convex, set $S$ are the points
in $S$ that cannot be written as convex combination of others.
Given some set $S = \{ (x,y) : \ldots \}$, we use $\proj_x S$ to denote the
projection of $S$ to the $x$-space, that is, $\proj_x S = \{ x : \exists y,
(x,y) \in S\}$.
If $g : C \subseteq \mathbb{R}^n \to \mathbb{R}$ is a function and $D \subseteq
C$, then we denote by $g^{vex}_D$ a convex underestimator of $g$ over $D$.
When $g$ is convex, $\partial g(x)$ denotes the subdifferential of $g$ at $x$.\\
Given an interval $I \subseteq \mathbb{R}$ and an arbitrary set $A \subseteq
\mathbb{R}^n$ we denote by $IA$ the set $\{ \lambda x \st \lambda \in I, x \in
A \}$.
Likewise, for $x \in \mathbb{R}^n$, $Ix := \{ \lambda x \st \lambda \in I \}$.
Given an integer $d$, we denote by $\mathcal{S}^d_+$ the cone of positive
semi-definite matrices of size $d$.\\
Finally, we use interchangeably the dot product notation $\langle c, x \rangle$
and $c^\T x$.

\section{Visible points and the reverse polar}
\label{sec:reverse}

In this section we introduce the concept of visible points and reverse polar,
and state some basic properties about them, which we will use in the rest of the
paper.
The main result in this section is that the reverse polar of the visible points
of a set is the reverse polar of the set (Proposition~\ref{prop:visible_reverse}).

Unless stated otherwise, we will assume $\xlp = 0$.
This is without loss of generality, since we can always translate the set $S$ to
$S - \xlp$.
We start by restating the definition of reverse polar.
\begin{definition}
  Let $S \subseteq \mathbb{R}^n$ and $\xlp \in \mathbb{R}^n$.
  The \emph{reverse polar} of $S$ at $\xlp$ is
  \[
    S^{\xlp} = \{ \alpha \in \mathbb{R}^n : \alpha^\T (x - \xlp) \geq 1, \text{
    for all } x \in S\}.
  \]
\end{definition}

As stated in the introduction, the reverse polar contains all cuts that separate
$\xlp$ from $S$.
\begin{definition}
  Let $S, V \subseteq \mathbb{R}^n$ and $\xlp \in \mathbb{R}^n$.
  We say that $V$ is a \emph{generator} of $S^{\xlp}$ if and only if
  \[
    V^{\xlp} = S^{\xlp}
  \]
\end{definition}
%
%
\begin{definition} \label{def:visible}
  Let $S \subseteq \mathbb{R}^n$ be closed and $\xlp \notin S$.
  The set of \emph{visible points} of $S$ from $\xlp$ is
  \begin{align*}
    V_S(\xlp)
      &=
      \{ x \in S : (x + [0,1](\xlp - x)) \cap S = \{x\} \} \\
      &=
      \{ x \in S : (x + (0,1](\xlp - x)) \cap S = \emptyset \}.
  \end{align*}
  We denote $V_S(0)$ by $V_S$ and note that
  \[
    V_S = \{ x \in S : [0, 1]x \cap S = \{x\} \}
    = \{ x \in S : [0, 1)x \cap S = \emptyset \}.
  \]
\end{definition}

The following concept is, in some sense, the opposite of the visible points.
\begin{definition}
    Let $S \subseteq \mathbb{R}^n$ be closed.
    The \emph{shadow} of $S$ from 0 is
    \[
        \shw S = [1, \infty)S.
    \]
\end{definition}
The concept of shadow has also been called \emph{penumbra}~\cite[p.
22]{Rockafellar1970},\cite{TindWolsey1982,ConfortiWolsey2018} and \emph{aureole
closure}~\cite{Ruys1974}.
The following are some basic properties of the reverse polar.
\begin{lemma}{\cite[Property 9.2.2]{Ruys1974}} \label{lemma:ruys}
  Let $S, T \subseteq \mathbb{R}^n$.
  Then,
  \begin{enumerate}
    \item $S^0 = (\shw S)^0 = (\conv S)^0 = (\cl S)^0$.
    \item $S^0 = \emptyset$ if and only if $0 \in \clconv S$.
    \item $S \subseteq T$ implies $T^0 \subseteq S^0$.
    \item If $0 \notin \clconv S$, then $(S^0)^0 = \shw \clconv S$.
  \end{enumerate}
\end{lemma}

We will now show that $V_S$ is a generator of $S^0$.
To this end, we need the following lemma, which says that the shadow of what can
be seen of a set is the same as the shadow of the whole set.
Likewise, what can be seen of a set is the same as what can be seen of the
shadows of the set.
\begin{lemma} \label{lemma:shadow_visible_invariant}
  Let $S \subseteq \mathbb{R}^n$ be a closed set such that $0 \notin S$.
  Then, $\shw V_S = \shw S$ and $V_{\shw S} = V_S$.
\end{lemma}
\begin{proof}
  First we prove that $\shw V_S = \shw S$.
  Clearly, $\shw V_S \subseteq \shw S$.

  Let $y \in \shw S$, then $y = \lambda x$ with $x \in S$, $\lambda \geq 1$.
  Let $I = \{ \mu \geq 0: \mu x \in S \}$ and $\mu_0 = \min I$.
  The minimum exists since $I$ is closed and not empty as $S$ is closed and
  $1\in I$, respectively.
  From $1 \in I$, we deduce $\mu_0 \leq 1$, and from $0 \notin S$,
  $\mu_0 > 0$.
  Hence, $\mu_0 x \in V_S$ and $y = \tfrac{\lambda}{\mu_0} (\mu_0 x)
  \in \shw V_S$, since $\tfrac{\lambda}{\mu_0} \geq 1$.

  Now we prove that $V_{\shw S} = V_S$.
  Clearly, $S \subseteq \shw S$ implies that $V_S \subseteq V_{\shw S}$.

  Let $x_0 \in V_{\shw S}$.
  Then $x_0 \in \shw S$, so there exists $\lambda \geq 1$ and $x \in S$
  such that $x_0 = \lambda x$.
  Note that $\lambda = 1$, since otherwise, $\frac{1}{\lambda} x_0 = x \in
  S \subseteq \shw S$ which cannot be as $x_0$ is visible.
  Thus, $x_0 \in V_S$.
\end{proof}

\begin{proposition} \label{prop:visible_reverse}
  Let $S \subseteq \mathbb{R}^n$ be a closed set.
  Then,
  \[
    (S \cap V_S)^0 = V_S^0 = S^0.
  \]
\end{proposition}
\begin{proof}
  The first equality just comes from the fact that $V_S \subseteq S$.

  If $0 \in S$, then the equality holds as all the sets are empty.
  Otherwise, the equality follows from $V_S^0 = (\shw V_S)^0 = (\shw S)^0
  = S^0$, where the first and last equalities are by Lemma~\ref{lemma:ruys} and the
  middle one, by Lemma~\ref{lemma:shadow_visible_invariant}.
\end{proof}

\section{The smallest generators}
\label{sec:smallest}

\subsection{Motivation}

In the previous section we showed that there is a set $U \subseteq S$ such that
$(U \cap S)^0 = S^0$, namely, $U = V_S$.
This set can be used to improve separation routines as was shown already in
Example~\ref{intro:ex1}.
We will come back to applications of the visible points to separation in the
next section.

The topic of this section is motivated by the following example, where the set
$V_S$ is much larger than the smallest generator.
\begin{figure}[!h]
  \centering
  \includegraphics[width=.3\linewidth]{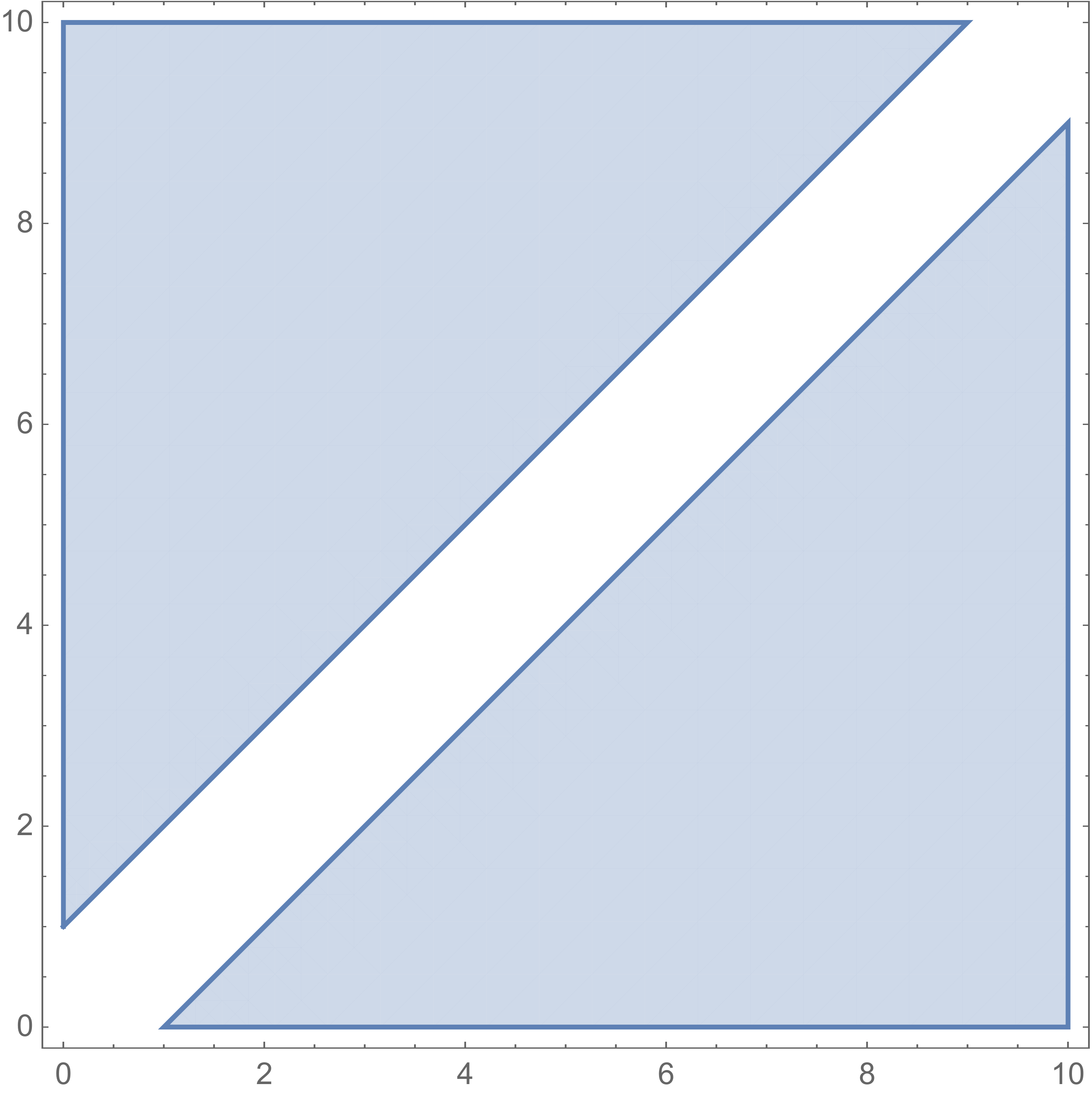}
  \includegraphics[width=.3\linewidth]{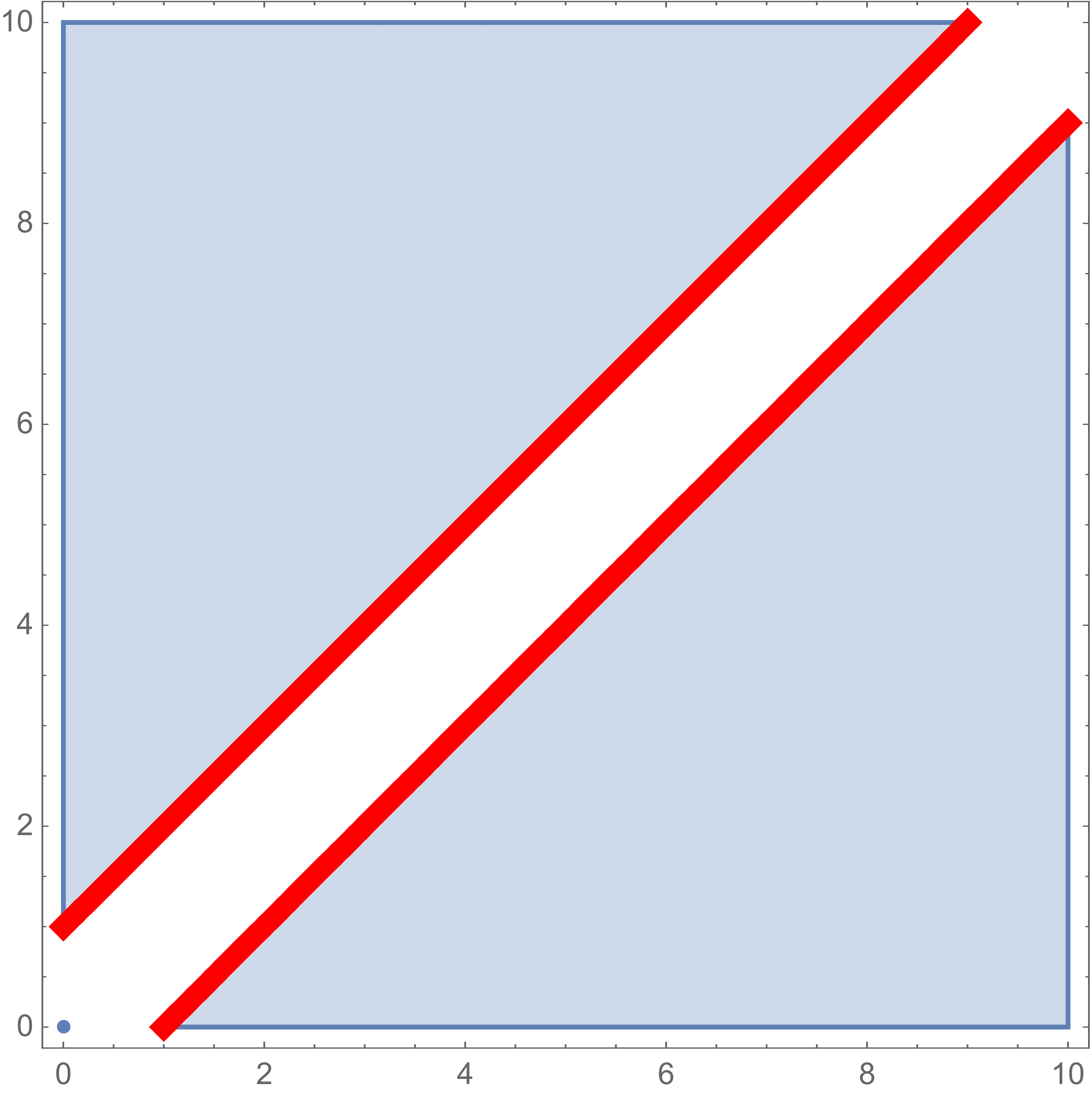}
  \includegraphics[width=.3\linewidth]{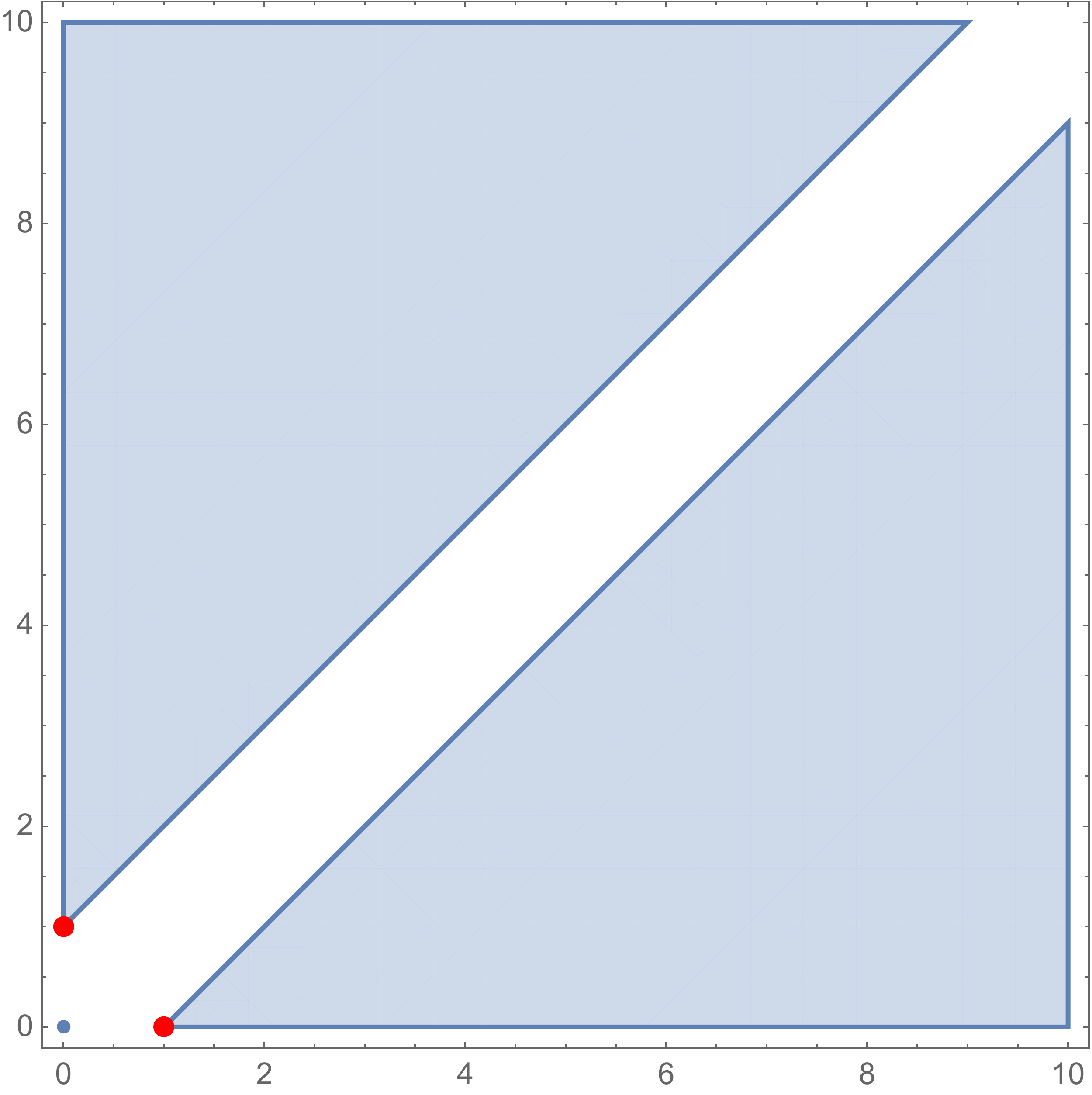}
  \caption{The region $S$.
    In the middle picture $V_S$ are the points described by the thick red
    line.
    In the right picture the red points form the smallest set $V$ such that $V^0
    = S^0$.
  }
  \label{fig:ex2}
\end{figure}
\begin{example}
  Consider the constrained set $S= \{(x_1, x_2) \in \mathbb{R}_+^2 \st (x_1
  - x_2)^2 \geq 1 \}$ depicted in Figure~\ref{fig:ex2}.
  The visible points are the lines $x_2 = x_1 + 1$ and $x_2 = x_1 -1$
  intersected with the first orthant.
  However, it is not hard to see that $V = \{(0,1), (1,0)\}$ is the smallest
  \emph{closed} generator of $S^0$.
\end{example}
This example motivates the following question.
\begin{question} \label{q:smallest}
  What is, if any, the smallest closed set $U$ such that $U^0 = S^0$?
\end{question}
The reason we restrict to generators that are closed sets is to avoid
representation issues.
For example, if $S$ is the ball of radius 1 centered at $(2,0)$, then
Theorem~\ref{thm:qvis} implies that the left arc joining $(2,1)$ and $(2,-1)$ generates
$S^0$.
However, the rational points on this arc also generate $S^0$ and the smallest
set generating $S^0$ does not exist.
In order to avoid such issues, we concentrate on closed generators.

As can be seen from simple examples, such as $S = \mathbb{R}_+ \times \{1\}$ for
which every $a \geq 0$ defines the generator $(\{0\} \cup [a,\infty)) \times
\{1\}$, the smallest closed generator must not exist.
However, a smallest closed convex generator might exist and so we ask the
following question.
\begin{question} \label{q:smallest_convex}
  What is, if any, the smallest closed convex generator of $S^0$?
\end{question}

We are mainly interested in applying our results to the separation problem, as
already explained in the introduction.
In that case, the set $S$ usually looks like $S = C \cap F$, where $C$ is
a convex set and $F$ is the sublevel set of some non-convex function, see the
next section.
In this context, replacing $C$ by a smaller convex set might be beneficial for
the separation problem (see Example~\ref{ex:quad}).
Thus, it is also natural to consider the following question.
\begin{question} \label{q:inter}
  What is, if any, the smallest closed convex set $U$ such that $S \cap U$
  generates $S^0$?
\end{question}

The last two questions are not the same.
Informally, $S$ is only used to define $S^0$ in Question~\ref{q:smallest_convex}, and so
any other set $T$ such that $T^0 = S^0$ can be used to formulate the question.
For instance, we can assume without loss of generality that $S$ is closed and
convex, since Lemma~\ref{lemma:ruys} implies that $(\clconv S)^0 = S^0$.
In contrast, in Question~\ref{q:inter} we are asking for the smallest generator
contained in $S$.

As we will see, the answer to Question~\ref{q:smallest_convex} is that $\clconv
V_{\clconv S}$ is the smallest closed convex generator of $S^0$.
However, the next two examples show that Question~\ref{q:inter} is a bit more delicate.

The first example shows that, in general, there is no unique smallest closed
convex set $U$ such that $(S \cap U)^0 = S^0$.
\begin{example}
  Let $S = \{ (1,0), (0,1), (-1,0), (0,-1) \}$.
  Since $0 \in \conv S$, $S^0 = \emptyset$.

  Clearly $V = \{0\} = V_{\conv S}$ is the smallest closed convex set such that
  $V^0 = \emptyset$.
  However, $S \cap V = \emptyset$, which implies that $(S \cap V)^0
  = \mathbb{R}^2 \neq S^0$.
  Furthermore, $U_1 = \{ (\lambda, 0) \st \lambda \in [-1,1] \}$ and $U_2 = \{
  (0, \lambda) \st \lambda \in [-1,1] \}$ are both closed convex and $(U_i \cap
  S)^0 = S^0$.
  Since $U_1 \not\subseteq U_2$ and $U_2 \not\subseteq U_1$ we conclude that
  there is no smallest closed convex set $U$ such that $(U \cap S)^0 = S^0$.
\end{example}

However, we cannot even expect to find a minimal closed convex set $U$ such that
$(S \cap U)^0 = S^0$.
\begin{example}
  Let $S = \{ (0,1) \} \cup \{(\lambda,2) \st \lambda \geq 0\}$.
  We have $S^0 = \{ \alpha \st \alpha_1 \geq 0, \alpha_2 \geq 1 \}$.

  Indeed, $(0,1) \in S$ implies that $\alpha_2 \geq 1$.
  If $\alpha_1 < 0$ for some $\alpha \in S^0$, then there is a large enough
  $\lambda$ such that $\lambda \alpha_1 + 2 \alpha_2 < 1$ and $(\lambda, 2) \in
  S$.
  On the other hand, if $\alpha_1 \geq 0$ and $\alpha_2 \geq 1$, then $\alpha_1
  x_1 + \alpha_2 x_2 \geq 1$ for every $(x,y) \in S$.

  Let $T_M = \{ (0,1) \} \cup \{(\lambda,2) \st \lambda \geq M\}$ and $U_M
  = \clconv( T_M)$.
  The same argument as above shows that $(U_M \cap S)^0 = T_M^0 = S^0$.
  Notice that any $U$ with $(U \cap S)^0 = S^0$ must contain a sequence
  $\lambda_n \to \infty$ such that $(\lambda_n,2) \in S$.
  Thus, any minimal $U$, if it exists, must be of the form $U_M$ for some $M
  \geq 0$.

  It is clear that $U_{M_1} \subseteq U_{M_2}$ if and only if $M_1 > M_2$ and
  $\bigcap_{M > 0} U_M = \{(\lambda,1) \st \lambda \geq 0\}$.
  However, $S \cap \{(\lambda,1) \st \lambda \geq 0\} = \{(0,1)\}$ and
  $\{(0,1)\}^0 \neq S^0$.
  Therefore, there is no minimal $U$.

  On the other hand, $V = \{(\lambda,1) \st \lambda \geq 0\} = V_{\clconv S}$ is
  the smallest closed convex set such that $V^0 = S^0$.
\end{example}

However, these are the only ``pathological cases''.
Indeed, as we will see, if $\conv(S)$ is closed (e.g. when $S$ is compact) and
$0 \notin \conv S$, (i.e., $S^0 \neq \emptyset$), then $\clconv V_{\conv S}$ is
the smallest closed convex set such that $\clconv V_{\conv S} \cap S$ generates
$S^0$.
\begin{remark} \label{rmk:example_deutsch}
  The closure operations are needed because, in general, $V_S$ and $\conv V_S$
  are not closed, even when $S$ is convex and compact.
  Indeed, it is shown in~\cite[Example 15.5]{DeutschHundalZikatanov2013} that
  for
  \[
    S := (1,0,0) +
    \cone\{(1,\alpha,\beta) \st \alpha^2 + (\beta - 1)^2 \leq 1\},
  \]
  $V_S$ is open.
  The authors show that the points $(2, \sin(t), 1 + \cos(t))$ are visible for
  $t \in (0, \pi)$, but the limit when $t$ approaches $\pi$, (2,0,0), is
  not.
  The remark follows from a modification of this example so that $S$ is compact,
  e.g., by intersecting it with $[0,3]\times\mathbb{R}^2$.
\end{remark}

\subsection{Preliminaries}
Here we collect a few lemmata that we are going to need in order to answer
Questions~\ref{q:smallest},~\ref{q:smallest_convex}~and~\ref{q:inter}.
\begin{lemma}{\cite[Proposition 15.19]{DeutschHundalZikatanov2013}}
  \label{lemma:deutsch}
  Let $S$ be a closed convex set such that $0 \notin S$.
  If $x \in V_S$ is a strict convex combination of $x_1, \ldots x_m \in S$, then
  $x_1, \ldots x_m \in V_S$.
\end{lemma}
This result immediately implies the following two lemmata.
\begin{lemma} \label{lemma:char_extreme_visible}
  Let $S \subseteq \mathbb{R}^n$ be a closed convex set such that $0 \notin S$.
  Then, $\ext V_S = V_S \cap \ext S$.
\end{lemma}
\begin{proof}
  We start by proving $\ext V_S \subseteq V_S \cap \ext S$.
  Let $x \in \ext V_{S}$.
  Clearly, $x \in V_S$.
  If $x \notin \ext S$, then there are $x_1, \ldots, x_m \in S$ such that $x$ is
  a strict convex combination of $x_1, \ldots, x_m$.
  Lemma~\ref{lemma:deutsch} implies that $x_i \in V_S$ for every $i = 1, \ldots,
  m$.
  Thus, $x$ is not an extreme point of $V_S$.
  This contradiction proves that $x \in \ext S$.

  If $x \in V_S \cap \ext S$ but $x \notin \ext V_S$, then $x$ is a strict
  convex combination of some elements of $V_S$.
  Since $V_S \subseteq S$, $x$ is a strict convex combination of some element of
  $S$.
  This is a contradiction with $x \in \ext S$.
\end{proof}

\begin{lemma} \label{lemma:visible_chull}
  Let $S \subseteq \mathbb{R}^n$ be closed set such that $\conv S$ is closed and
  $0 \notin \conv S$.
  Then,
  \[
    \conv(V_{\conv S}) = \conv(S \cap V_{\conv S}).
  \]
\end{lemma}
\begin{proof}
  From $S \cap V_{\conv S } \subseteq V_{\conv S}$, it follows that
  $\conv(S \cap V_{\conv S}) \subseteq \conv(V_{\conv S})$.

  To prove the other inclusion it is enough to show that
  $V_{\conv S} \subseteq \conv(S \cap V_{\conv S})$.
  Let $x \in V_{\conv S}$.
  Then, $x \in \conv(S)$ and so $x$ is a strict convex combination of some points
  of $x_1, \ldots, x_m \in S$.
  Then, by Lemma~\ref{lemma:extreme_points}, $x_1, \ldots, x_m \in S \cap
  V_{\conv S}$.
  Thus, $x \in \conv(S \cap V_{\conv S})$.
\end{proof}

We remark that the previous lemma does not follow from
Lemma~\ref{lemma:char_extreme_visible} by just taking the convex hull operation to
the equality, since $\conv S$ may not have extreme points.
%

The following is a slight extension of \cite[Corollary 18.3.1]{Rockafellar1970}.
\begin{lemma} \label{lemma:extreme_points}
  Let $S \subseteq \mathbb{R}^n$ be a closed set.
  Then, $\ext \clconv {S} \subseteq  S$.
\end{lemma}
\begin{proof}
  Recall that $x_0$ is an exposed point~\cite[Section 18]{Rockafellar1970} of
  a closed convex set $C$ if and only if there exists an $\alpha$ such that
  $\{x_0\} = \argmax_{x \in C} \alpha^\T x_0$.

  We will show that the exposed points of $\clconv S$ is a subset of $S$.
  Then, by Straszewicz's Theorem~\cite[Theorem 18.6]{Rockafellar1970} and the
  closedness of $S$, it follows that $\ext \clconv {S} \subseteq  S$.
  Note that when the set of exposed points is empty, the result follows trivially.
  Thus, we assume that the set of exposed points is non-empty.

  Let $x_0$ be an exposed point of $\clconv S$ and let $\alpha$ be a direction
  that exposes it.
  Then, $\sup_{x \in S} \alpha^\T x = \alpha^\T x_0$.
  Since $S$ is closed, there exists $x_1 \in S$ such that $\alpha^\T x_1
  = \alpha^\T x_0$.
  However, since $x_1 \in S \subseteq \clconv S$ and $\alpha$ exposes $x_0$, we
  must have $x_1 = x_0$.
  Thus, $x_0 \in S$.
\end{proof}

\subsection{Results}

Let us start by answering Question~\ref{q:smallest_convex}.
\begin{theorem} \label{thm:smallest_convex}
  Let $S \subseteq \mathbb{R}^n$ be closed.
  Then,
  \[
    (\clconv V_{\clconv S})^0  = S^0.
  \]
  Furthermore, if $C \subseteq \mathbb{R}^n$ is a closed convex generator of
  $S^0$, then
  \[
    \clconv(V_{\clconv S}) \subseteq C.
  \]
\end{theorem}
\begin{proof}
  Note that if $S^0 = \emptyset$, then $0 \in \clconv S$ and $V_{\clconv S}
  = \{0\}$, from which the theorem clearly follows.
  Thus, we assume $S^0 \neq \emptyset$.

  Lemma~\ref{lemma:ruys} implies that $(\clconv V_{\clconv S})^0 = (V_{\clconv S})^0$
  and $S^0 = (\clconv S)^0$.
  Proposition~\ref{prop:visible_reverse} implies $(\clconv S)^0 = (V_{\clconv S})^0$.

  To show the second statement of the theorem, let $C$ be closed and convex such
  that $C^0 = S^0$.
  Since $C$ is closed and convex, it is enough to prove that $V_{\clconv(S)}
  \subseteq C$.
  Suppose, by contradiction, that this is not the case, i.e., there is an
  $\bar{x} \in V_{\clconv(S)}$ such that $\bar{x} \notin C$.
  There are two cases, either $[0,1]\bar{x} \cap C = \emptyset$ or $[0,1]\bar{x}
  \cap C \neq \emptyset$.
  We will deduce a contradiction from each of them.

  First, suppose $[0,1]\bar{x} \cap C = \emptyset$.
  Both sets are closed and $[0,1]\bar{x}$ is bounded, thus, they can be
  separated.
  Indeed, $0 \in [0,1]\bar{x}$ and~\cite[Corollary 11.4.1]{Rockafellar1970}
  ensure the existence of $\alpha$ such that $\alpha x \geq 1$ for every $x \in
  C$ and $\alpha \bar{x} < 1$.
  This means that $\alpha \in C^0$.
  However, $\alpha \notin (\clconv S)^0 = S^0$, since $\bar{x} \in \clconv S$.
  This contradicts $S^0 = C^0$.

  Now, suppose $[0,1]\bar{x} \cap C \neq \emptyset$.
  Since $0, \bar{x} \notin C$, there must be $\mu \in (0,1)$ such that $\mu
  \bar{x} \in C$.
  However, $\bar{x} \in V_{\clconv(S)}$ implies that $\mu \bar{x} \notin
  \clconv(S)$.
  Thus, the same argument as above ensures that $\mu \bar{x}$ can be separated
  from $\clconv(S)$.
  Therefore, there is an $\alpha$ such that $\alpha^\T x \geq 1$ for every $x
  \in \clconv(S)$ while $\alpha^\T \mu\bar{x} < 1$.
  Hence, $\alpha \in S^0$ and the contradiction follows from the fact that
  $\mu\bar{x} \in C$ implies $\alpha \notin C^0$.

  Therefore, we conclude that $\clconv V_{\clconv S} \subseteq C$.
\end{proof}

Now we show that if $\conv S$ is closed and $0 \notin \conv S$, then $\clconv
V_{\conv S}$ is the answer to Question~\ref{q:inter}, i.e., is \emph{the} smallest
closed convex $U$ such that $(U \cap S)^0 = S^0$.
\begin{theorem} \label{thm:smallest_inter}
  Let $S \subseteq \mathbb{R}^n$ be a closed set such that $\conv S$ is closed
  and $0 \notin \conv S$, i.e., $S^0 \neq \emptyset$.
  Then,
  \[
    (\clconv(V_{\conv S}) \cap S)^0  = S^0.
  \]
  Furthermore, if $C$ is closed and convex such that $(C \cap S)^0 = S^0$, then
  \[
    \clconv(V_{\conv S}) \subseteq C.
  \]
\end{theorem}
\begin{proof}
  We first show that $(\clconv(V_{\conv S}) \cap S)^0  = S^0$.\\
  Clearly,
  \[
    V_{\conv S} \cap S \subseteq \clconv(V_{\conv S}) \cap S \subseteq S.
  \]
  Lemma~\ref{lemma:ruys} implies that
  \[
    S^0 \subseteq (\clconv(V_{\conv S}) \cap S)^0 \subseteq (V_{\conv S} \cap
    S)^0.
  \]
  Thus, it is enough to show that $(V_{\conv S} \cap S)^0 = S^0$.
  This follows from
  \begin{align*}
    (S \cap V_{\conv S})^0
        &= (\conv(S \cap V_{\conv S}))^0 &\text{Lemma~\ref{lemma:ruys}}\\
        &= (\conv V_{\conv S})^0 &\text{Lemma~\ref{lemma:visible_chull}}\\
        &= (V_{\conv S})^0 &\text{Lemma~\ref{lemma:ruys}}\\
        &= (\conv S)^0 &\text{Proposition~\ref{prop:visible_reverse}}\\
        &= S^0. &\text{Lemma~\ref{lemma:ruys}}\\
  \end{align*}

  To show the second statement of the theorem, let $C$ be a closed convex set
  such that $(C \cap S)^0 = S^0$.
  Lemma~\ref{lemma:ruys} implies that $(C \cap S)^0 = (\clconv(C \cap S))^0$.
  Theorem~\ref{thm:smallest_convex} implies that $\clconv(V_{\conv S}) \subseteq \clconv(C
  \cap S).$
  Clearly, $V_{\conv S} \subseteq \clconv(V_{\conv S})$ and $\clconv(C \cap S)
  \subseteq C \cap \conv S$.
  Therefore, $V_{\conv S} \subseteq C \cap \conv S$ which implies $V_{\conv S}
  \subseteq C$ as we wanted.
\end{proof}

Finally, we answer Question~\ref{q:smallest} in the case where $S$ is compact.
\begin{theorem} \label{thm:smallest_closed}
  Let $S$ be any closed set such that $0 \notin \clconv S$.
  If $D$ is any closed generator of $S^0$, then
  \[
    \clext V_{\clconv S} \subseteq D.
  \]

  If, in addition, $S$ is compact, then $\clext V_{\conv S}$ is the smallest
  closed generator of $S^0$.
\end{theorem}
\begin{proof}
  First, by Lemma~\ref{lemma:ruys} and $D^0 = S^0$, we have $\shw \clconv
  D = \shw \clconv S$.
  Then, Lemma~\ref{lemma:shadow_visible_invariant} implies that $V_{\clconv D}
  = V_{\clconv S}$.
  Hence, $\ext V_{\clconv D} = \ext V_{\clconv S}$.
  Therefore, $\ext V_{\clconv S} = \ext V_{\clconv D} \subseteq \ext \clconv
  D \subseteq D$, where the first and second containments are due to
  Lemma~\ref{lemma:char_extreme_visible} and Lemma~\ref{lemma:extreme_points},
  respectively.

  To prove the second statement, by Lemma~\ref{lemma:ruys}, it is enough to show that
  $\ext V_{\conv S}^0 = S^0$.
  First, as $\ext V_{\conv S} \subseteq \conv S$, we have $S^0 \subset (\ext
  V_{\conv S})^0$.

  To prove the other containment take any $\alpha \in (\ext V_{\conv S})^0$.
  Let $x \in \conv S$ be arbitrary.
  We will prove that $\alpha^\T x \geq 1$.
  This will imply that $\alpha \in (\conv S)^0 = S^0$ and, therefore, that
  $(\ext V_{\conv S})^0 \subseteq S^0$.\\
  Let $\lambda \in (0,1]$ be such that $\lambda x \in V_{\conv S}$.
  If $\lambda x \in \ext V_{\conv S}$, then $\alpha^\T \lambda x \geq 1$, which
  implies that $\alpha^\T x \geq \frac{1}{\lambda} \geq 1$.\\
  Now, assume $\lambda x \notin \ext V_{\conv S}$.
  Since $S$ is compact, $\conv S$ is closed and we can use
  Lemma~\ref{lemma:char_extreme_visible} to obtain that $\ext V_{\conv S} = V_{\conv
  S} \cap \ext \conv S$.
  Thus, $\lambda x \notin \ext \conv S$.
  Also by the compactness of $S$, \cite[Theorem 18.5.1]{Rockafellar1970} implies
  that $\lambda x$ is a strict convex combination of some $x_1, \ldots, x_m \in
  \ext \conv S$.\\
  Lemma~\ref{lemma:deutsch} implies that $x_1, \ldots, x_m \in V_{\conv S}$ and so
  Lemma~\ref{lemma:char_extreme_visible} implies that $x_1, \ldots, x_m \in \ext
  V_{\conv S}$.
  Since $\alpha \in (\ext V_{\conv S})^0$, it follows $\alpha^\T x_i \geq 1$ for
  every $i = 1, \ldots, m$.
  Hence, $\alpha^\T \lambda x \geq 1$ and, as before, $\alpha^\T x \geq
  \frac{1}{\lambda} \geq 1$.
\end{proof}

We remark that the closure operation is needed since the extreme points of
a set, in general, do not form a closed set, see~\cite[p. 167]{Rockafellar1970}.

\section{Applications to MINLP} \label{sec:minlp}

%
Here we apply the results from Section~\ref{sec:reverse} to MINLP.

In this section, unless specified otherwise, $\xlp \in \mathbb{R}^n$, $C$ is
a closed convex set that contains $\xlp$, and $S :=  \{ x \in C \st g(x) \leq
0 \}$, where $g: C \to \mathbb{R}$ is continuous and $g(\xlp) > 0$.
The idea is that $C$ represents a convex relaxation of our MINLP and $\xlp \in
C$ is the current relaxation solution that is infeasible for a constraint $g(x)
\leq 0$.

The basic scheme for applying our results is the following translation of
Observation~\ref{intro:observation}.
\begin{proposition} \label{prop:minlp:generic}
  Let $D \subseteq C$ be such that $(D \cap S)^{\xlp} = S^{\xlp}$, and $T = \{ x \in D \st
  g(x) \leq 0 \}$.
  If $\alpha^\T (x - \xlp) \geq 1$ is a valid inequality for $T$, then it is valid
  for $S$.
\end{proposition}
\begin{proof}
  Directly from $\alpha \in T^{\xlp} = (D \cap S)^{\xlp} = S^{\xlp}$.
\end{proof}

Of course, the applicability of the previous proposition relies on our ability
to obtain an easy-to-compute set $D$ that satisfies the hypothesis.
As shown in Section~\ref{sec:smallest}, $D = \clext \conv V_{\conv S}(\xlp)$
is the smallest we can hope for, but it is useless from a practical point of
view.
Instead, the set of visible points of $S$ (or a set enclosing them) is,
computationally, a better candidate as we will see in
Section~\ref{subsec:characterization_minlp}.
\begin{corollary} \label{cor:minlp:visibles_work}
  Let $D \subseteq C$ be such that $V_S(\xlp) \subseteq D$, and $T = \{ x \in D \st g(x)
  \leq 0 \}$.
  If $\alpha^\T (x - \xlp) \geq 1$ is a valid inequality for $T$, then it is valid for $S$.
\end{corollary}
\begin{proof}
  Clearly, $V_S(\xlp) \subseteq T = D \cap S \subseteq S$.
  The inclusion-reversing property of the reverse polar implies that $S^{\xlp}
  \subseteq (D \cap S)^{\xlp} \subseteq V_S(\xlp)^{\xlp} = S^{\xlp}$, where the last equality follows
  from Proposition~\ref{prop:visible_reverse}.
  The statement follows from Proposition~\ref{prop:minlp:generic}.
\end{proof}

In the context of separation via convex underestimators
Corollary~\ref{cor:minlp:visibles_work} reads
\begin{corollary} \label{cor:minlp:practical}
  Let $D \subseteq C$ be a closed convex set such that $V_S(\xlp) \subseteq D$, and
  let $T = \{ x \in D \st g(x) \leq 0 \}$.
  If $g^{vex}_D(\xlp) > 0$ and $\partial g^{vex}_D(\xlp) \neq \emptyset$, then a
  gradient cut of $g^{vex}_D$ at $\xlp$ is valid for $S$.
\end{corollary}
\begin{proof}
  Let $T_r := \{ x \in D \st g^{vex}_D(x) \leq 0 \}$ and $v \in \partial
  g^{vex}_D(0)$.
  The cut $g^{vex}_D(0) + v^\T x \leq 0$ is valid for $T_r$, and separates $0$
  from $T_r$.
  Since $T_r$ is a relaxation, i.e. $T \subseteq T_r$, it follows that the cut
  is also valid for $T$, and Corollary~\ref{cor:minlp:visibles_work} implies its
  validity for $S$.
\end{proof}

The previous result tells us that if we find a box, tighter than the bounds,
that contains the visible points, then we might be able to construct tighter
underestimators.
However, to compute a box containing $V_S(\xlp)$ we need to know how $V_S(\xlp)$
looks like.
That is the topic of the next section.

\subsection{Characterizing the visible points}
\label{subsec:characterization_minlp}
From the definition of visible points we have:
\begin{lemma} \label{lemma:characterization_visibility}
  Let $g : \mathbb{R}^n \to \mathbb{R}$ be a continuous function, $C \subseteq
  \mathbb{R}^n$ a closed convex set, and $S = \{ x \in C : g(x) \leq 0 \}$.
  If $\xlp \in C$ and $g(\xlp) > 0$, then
  \begin{equation} \label{eq:characterization_visibility}
    V_S(\xlp) = \left\{
      x \in C \st
      g(x) = 0,\
      g(x + \lambda(\xlp - x)) > 0 \text{ for every } \lambda
      \in (0,1]
    \right\}.
  \end{equation}
  Furthermore, if $g$ is differentiable, then every $x \in V_S(\xlp)$
  satisfies
  \[
    \langle \nabla g(x), \xlp - x \rangle \geq 0.
  \]
\end{lemma}
\begin{proof}
  Given that $\xlp \notin S$, by definition we have $x \in V_S(\xlp)$ if and only if
  $x \in S$ and for every $\lambda \in (0,1]$, $x + \lambda (\xlp - x)
  \notin C$ or $g(x + \lambda (\xlp - x)) > 0$.
  However, the convexity of $C$ and $\xlp \in C$ imply that for $x \in S$,
  $x + \lambda (\xlp - x) \in C$.
  Hence,
  \[
    V_S(\xlp) = \{ x \in C \st g(x) \leq 0,\ g(x + \lambda (\xlp - x))
    > 0 \text{ for every } \lambda \in [0,1) \}.
  \]
  Since $g$ is continuous, it follows that for $x \in V_S(\xlp)$,
  \[
    0 \geq g(x) = \lim_{\lambda \to 0^+} g(x + \lambda (x - \xlp)) \geq 0.
  \]
  Thus, $g(x) = 0$ which proves~\eqref{eq:characterization_visibility}.

  Now, assume that $g$ is differentiable and let $x \in V_S(\xlp)$.
  Then,
  \[
    0 \leq \lim_{\lambda \to 0^+}
    \frac{g(x + \lambda(\xlp - x))}{\lambda}
    = \lim_{\lambda \to 0^+} \frac{g(x + \lambda(\xlp - x)) - g(x)}{\lambda}
    = \langle \nabla g(x), \xlp - x \rangle.
  \]
  This concludes the claim.
\end{proof}
\begin{remark}
  Note that if we drop the hypothesis that $\xlp$ is in $C$, then there might be
  visible points for which $g$ is strictly negative, and there does not seem to
  be a nice description of the visible points.
  In such a case, $V_S(\xlp)$ would be a disjunctive set and we would even lose
  the valid (non-linear) inequality $\langle \nabla g(x), \xlp - x \rangle \geq
  0$.
  Likewise, if $C$ were not convex, or if we had more than one non-convex
  constraint, e.g., some variable has to be binary, then there does not seem to
  be a nice description of the visible points.
  This last point is rather unfortunate, it means that it might not be easy to
  generalize the technique to relaxations that involve more than one non-convex
  constraint.
  In particular, since a mixed-integer set usually consists of multiple
  non-convex constraints, the techniques presented here might not be applicable
  to MILPs.\\
  On the other hand, considering more constraints might allow us to see
  more of the feasible region.
  Therefore, in such cases one might have to try to use stronger generators such
  as $\clconv V_{\conv S}$, see also~\cite{VenkatachalamNtaimo2016}.
\end{remark}

\subsubsection{Quadratic constraints}

For quadratic constraints, the visible points have a particularly simple
description.
\begin{theorem} \label{thm:qvis}
  Let $C$ be a closed, convex set that contains $\xlp$.
  Let $g(x) = x^\T Q x + b^\T x + c$ and $S = \{ x \in C : g(x) \leq 0 \}$.
  If $g(\xlp) > 0$, then
  \begin{equation*}
    V_S(\xlp) = \left\{ x \in C \st g(x) = 0,\
    \langle \nabla g(\xlp), x \rangle + b^\T \xlp + 2c \geq 0 \right\}
  \end{equation*}
\end{theorem}
\begin{proof}
  $(\subseteq)$
  Let $x \in V_S(\xlp)$.
  By Lemma~\ref{lemma:characterization_visibility}, we have $g(x) = 0$ and
  $\langle \nabla g(x), \xlp - x \rangle \geq 0$.
  Equivalently,
  \begin{align*}
    x^\T Q x + b^\T x + c &= 0,\\
    2 x^\T Q (\xlp - x) + b^\T (\xlp - x)  &\geq 0.
  \end{align*}
  By multiplying the equation by 2, adding it to the inequality, and
  re-arranging terms we obtain the result.

  $(\supseteq)$
  Let $x$ satisfy $g(x) = 0$ and $\langle \nabla g(\xlp), x \rangle + b^\T
  \xlp + 2c \geq 0$.
  Then, subtracting $2g(x)$ from $\langle \nabla g(\xlp), x \rangle + b^\T
  \xlp + 2c \geq 0$ yields $\langle \nabla g(x), \xlp - x \rangle \geq 0$.
  Let
  \[
    q(\lambda) = g(x + \lambda (\xlp - x)) \text{, for } \lambda \in
    \mathbb{R}.
  \]
  The derivative is given by $q^{\prime}(\lambda) = \langle \nabla g (x
  + \lambda (\xlp - x)), \xlp - x \rangle$, and $q^{\prime}(0) = \langle \nabla
  g (x), \xlp - x \rangle \geq 0$.
  Since $q$ is quadratic, $q(1) = g(\xlp) > 0$, $q(0) = g(x) = 0$, and
  $q^{\prime}(0) \geq 0$, we have that $q$ has no roots in $(0,1]$.
  Thus, $g(x + \lambda (\xlp - x)) = q(\lambda) > 0$ for every $\lambda \in
  (0,1]$ and, from Lemma~\ref{lemma:characterization_visibility}, we conclude
  that $x \in V_S(\xlp)$ as we wanted.
\end{proof}

\begin{remark}
  Theorem~\ref{thm:qvis} implies in particular that the visible points of
  a closed convex set intersected with a quadratic constraint, from a point in
  the convex set, is always closed.
  This does not contradict~\cite[Example 15.5]{DeutschHundalZikatanov2013}
  mentioned in Remark~\ref{rmk:example_deutsch}.
  Indeed, if one represents the cone as a quadratic constraint $q(x) \leq 0$,
  then the origin must be feasible for the quadratic constraint.
  This is easily seen from the fact that the ray $[1, \infty) (1,0,0)$ is in the
  boundary of the cone, which implies that $q(\lambda,0,0) = 0$ for $\lambda
  \geq 0$.
  But $q(\lambda,0,0)$ is a univariate quadratic function and as such can have
  at most two roots if it is nonzero.
  Hence, $q(\lambda, 0, 0) = 0$ and, in particular, $q(0,0,0) = 0$.
\end{remark}

\begin{remark}
  The hyperplane $\langle \nabla g(\xlp), x \rangle + b^\T \xlp + 2c = 0$ is
  known as the \emph{polar hyperplane}~\cite{FasanoPesenti2017} of the point
  $\xlp$ with respect to the quadratic $g$ in projective geometry.
  In fact, homogenizing the quadratic $g$ yields the quadric
  \[
    g_h(x, x_0) = x^\T Q x + b^\T x x_0 + c x_0^2 =
    \begin{pmatrix}
      x \\
      x_0
    \end{pmatrix}^\T
    \begin{pmatrix}
      Q & \frac{b}{2} \\
      \frac{b^\T}{2} & c
    \end{pmatrix}
    \begin{pmatrix}
      x \\
      x_0
    \end{pmatrix}.
  \]
  The polar hyperplane of
  $
  \begin{pmatrix}
    \xlp \\
    1
  \end{pmatrix}
  $
  with respect to $g_h(x,x_0) = 0$ is then given by
  \begin{align*}
    &\langle \nabla g_h(x, x_0), (\xlp, 1) \rangle = 0
    \\ \iff &
    2\xlp^\T Q x + b^\T \xlp x_0 + b^\T x + 2c x_0 = 0.
  \end{align*}
  Intersecting with $x_0 = 1$ yields $\langle \nabla g(\xlp), x \rangle + b^\T
  \xlp + 2c = 0$.
\end{remark}

\begin{example} \label{ex:quad}
  Consider the function
  \[
    g(x_1, x_2, x_3) = - x_1 x_2 +x_1 x_3 + x_2 x_3 - x_1 - x_2 -x_3 + 1,
  \]
  the boxed domain $B = [-\frac{1}{10},2]\times[0,2]^2$, the constrained set
  \[
    S= \{x \in B \st g(x) \leq 0\},
  \]
  and the infeasible point $\xlp = (0,0,0)$.
  By Theorem~\ref{thm:qvis}, the visible points from $\xlp$ are given by
  \[
    V_S(\xlp) =
    \{
      (x_1, x_2, x_3) \in B \st g(x) = 0,\  x_1 + x_2 + x_3 \geq 0
    \},
  \]
  as shown in Figure~\ref{fig:ex:quad}.

  The tightest box bounding $V_S$ is
  \[
    R = [-\frac{1}{10}, 1] \times [0, \frac{1}{20}(23 + 3 \sqrt{5})] \times
    [0, \frac{1}{20}(19 + 3 \sqrt{5})].
  \]
  The linear underestimators of $g$ obtained by using
  McCormick~\cite{McCormick1976} inequalities for each term over $B$ and $R$ are
  \[
    1 \leq x_1 + 3 x_2 + \frac{11}{10} x_3 \text{ and }
    1 \leq x_1 + 2 x_2 + \frac{11}{10} x_3,
  \]
  respectively.
  Since $0 \leq x_2$, it follows that the underestimator over $R$ dominates the
  underestimator over $B$.
  We remark that the improvement in this particular cut is only due to the
  improvement on the upper bound of $x_1$.

  \begin{figure}[!h]
    \centering
    \includegraphics[width=.31\linewidth]{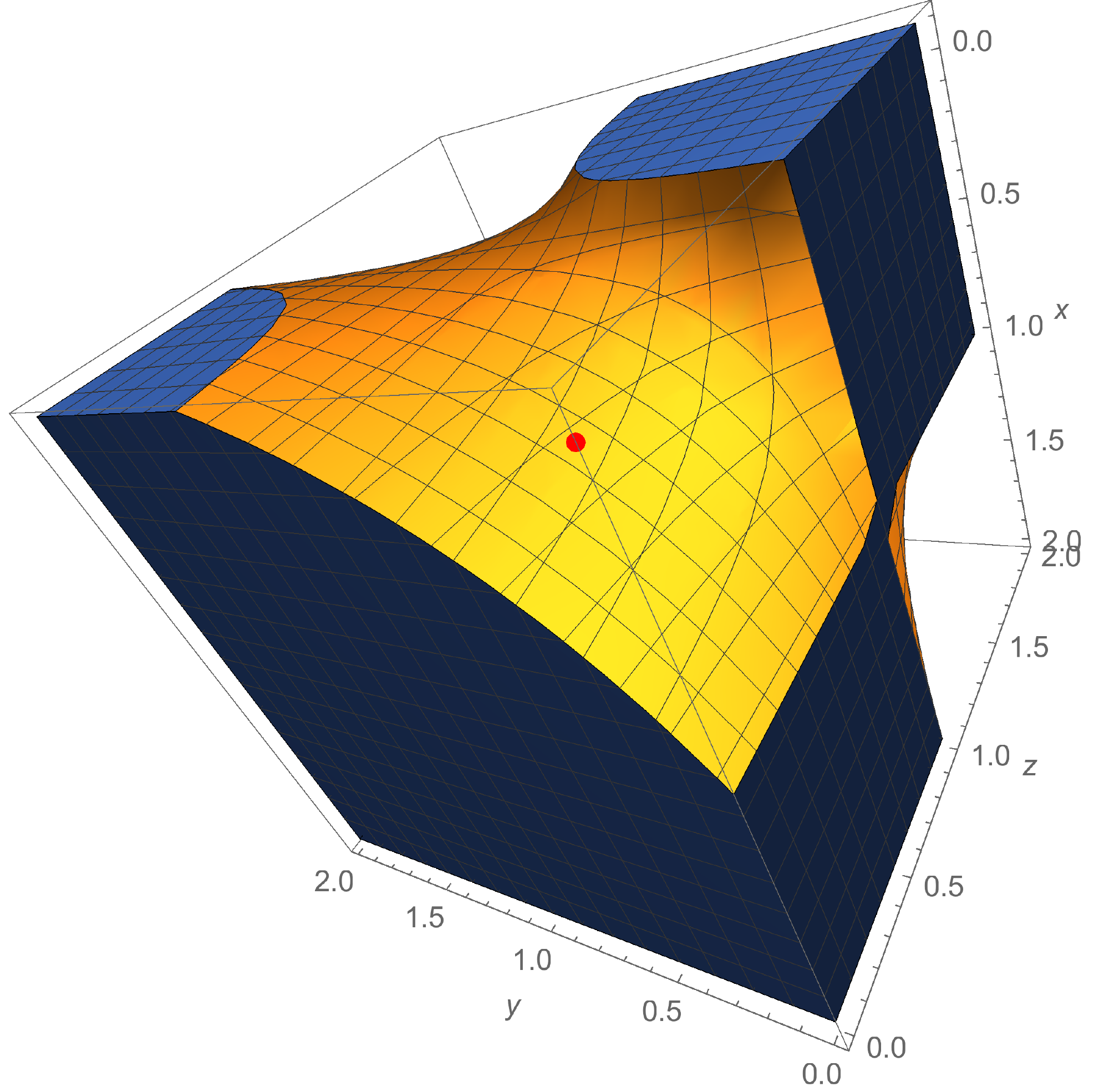}
    \includegraphics[width=.31\linewidth]{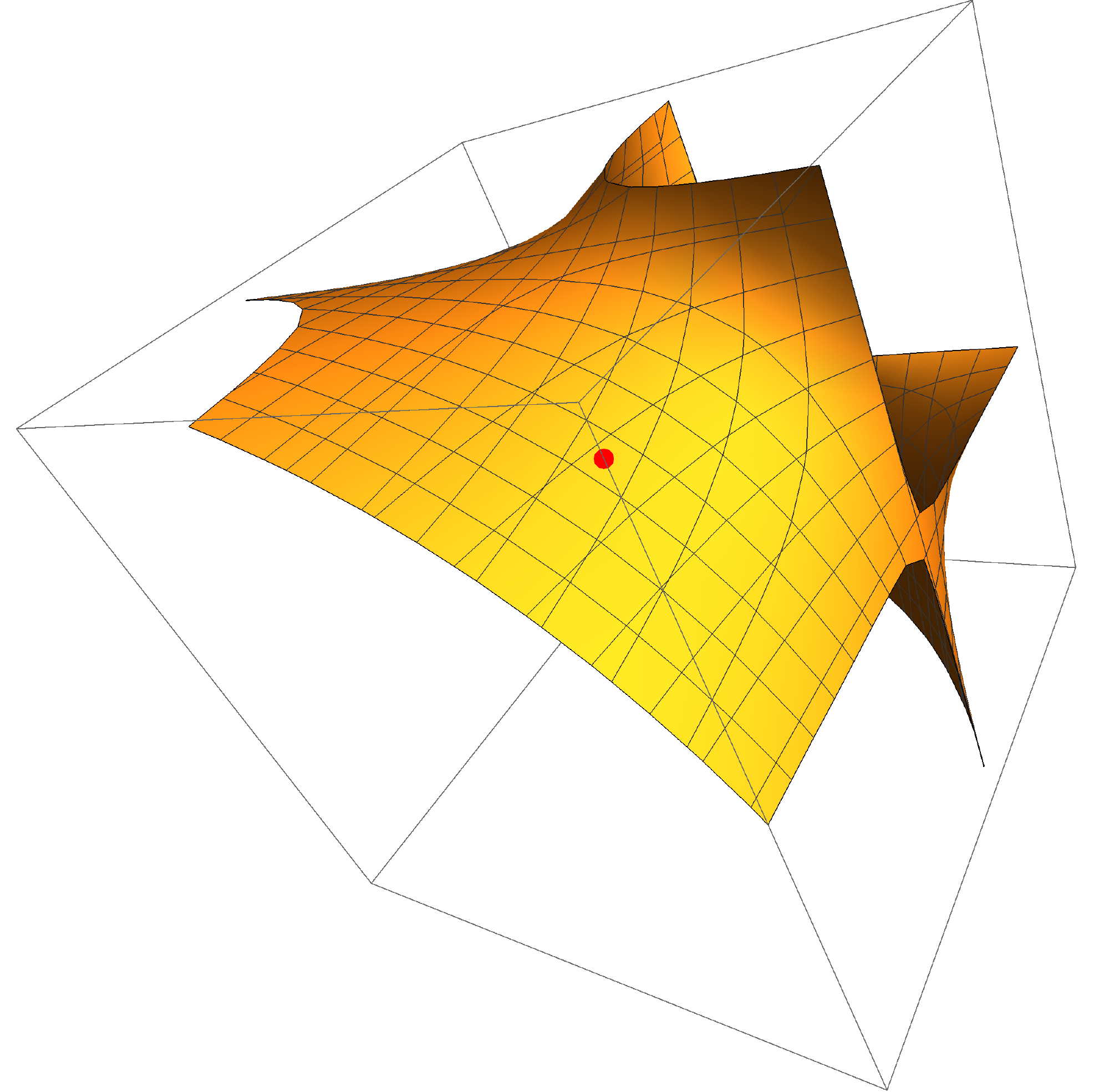}
    \includegraphics[width=.31\linewidth]{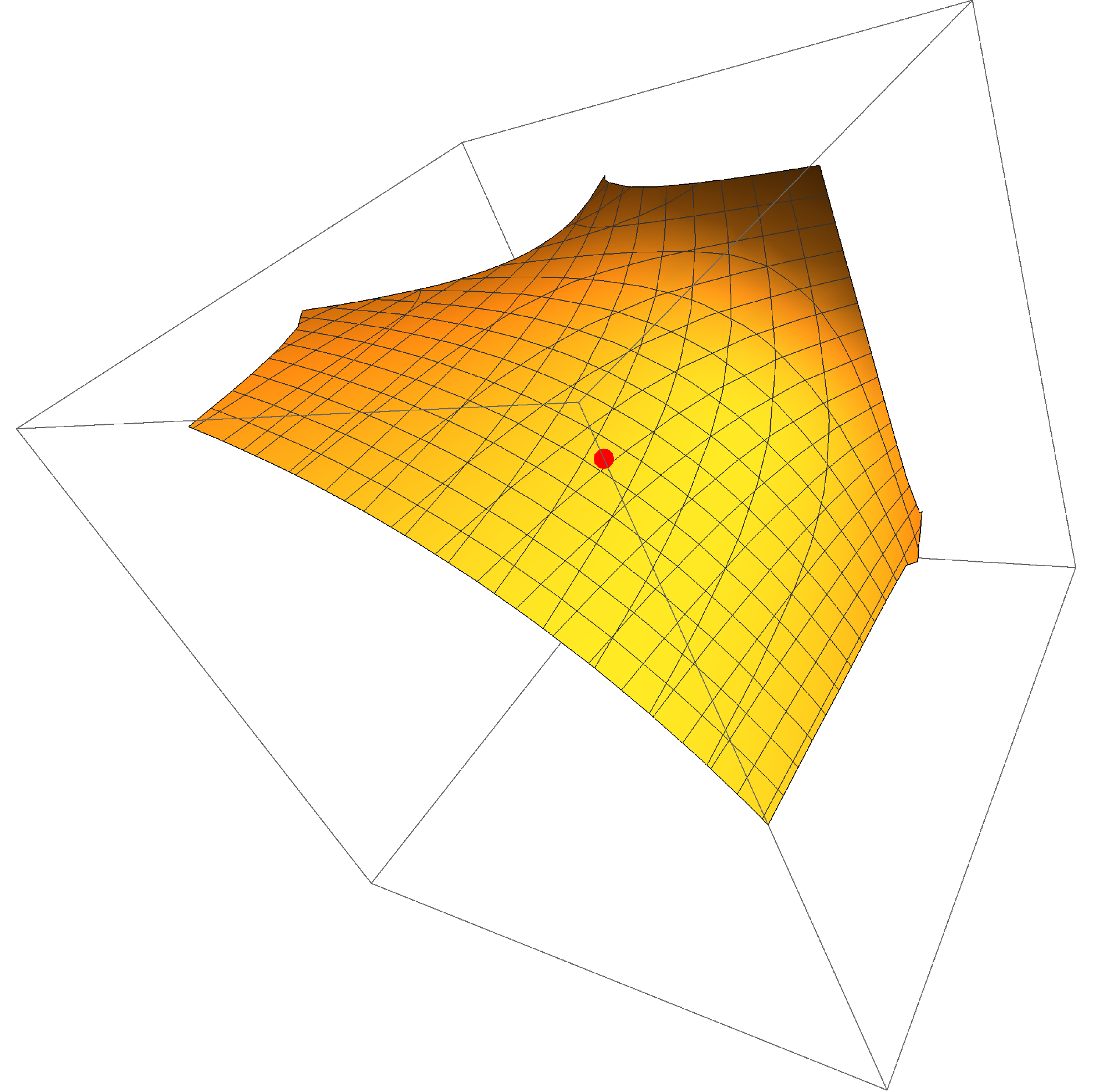}
    \caption{The left plot shows the feasible region $S$ and
      $\xlp$.
      The set $\{ x \in B \st g(x) = 0 \}$ appears in the middle plot.
      Finally, the visible points, $V_S(\xlp)$, are plotted on the right.
    }
    \label{fig:ex:quad}
  \end{figure}
\end{example}

\subsubsection{Polynomial constraints}

For a general polynomial $g$, the condition
\begin{equation} \label{eq:positive}
  g(x + \lambda(\bar x - x)) > 0 \text{ for every } \lambda \in (0,1]
\end{equation}
of \eqref{eq:characterization_visibility} asks for the univariate polynomial
$p_x(\lambda) = g(x + \lambda(\bar x - x))$ to be positive on $(0,1]$.
We can then use the theory of non-negative polynomials to translate a relaxation
of the infinitely many constraints \eqref{eq:positive} to a finite number of
constraints.
From the following characterization of non-negative polynomials on intervals we
can derive an extended formulation for the relaxation of
\eqref{eq:characterization_visibility},
\[
  R_S(\xlp) := \left\{
    x \in C \st
    g(x) = 0,\
    g(x + \lambda(\xlp - x)) \geq 0 \text{ for every } \lambda
    \in [0,1]
  \right\}.
\]
\begin{theorem} \label{thm:nonnegative}
  Let $p \in \mathbb{R}[\lambda]$ be a polynomial.
  Then $p$ is non-negative on $[0,1]$ if and only if
  \begin{enumerate}
    \item the degree of $p$ is $2d$ and there exist $s_1, s_2 \in
      \mathbb{R}[\lambda]$ of degree $d$ and $d-1$, respectively,
      such that
      \[
        p(\lambda) = s_1(\lambda)^2 + \lambda(1 - \lambda) s_2(\lambda)^2.
      \]

    \item the degree of $p$ is $2d+1$ and there exist $s_1, s_2 \in
      \mathbb{R}[\lambda]$ of degree $d$, such that
      \[
        p(\lambda) = \lambda s_1(\lambda)^2 + (1 - \lambda) s_2(\lambda)^2.
      \]
  \end{enumerate}
\end{theorem}
\begin{proof}
  See \cite{PowersReznick2000}.
\end{proof}

\begin{theorem} \label{thm:poly_vis}
  Let $C$ be a closed convex set that contains $\xlp$.
  Let $g(x)$ be a polynomial such that $g(\xlp) > 0$ and $S = \{
  x \in C : g(x) \leq 0 \}$.
  Let $p_x(\lambda) = g(x + \lambda(\xlp - x))$.
  \begin{enumerate}
    \item If the degree of $g$ is $2d$, then
      \begin{equation*}
        R_S(\xlp) = \proj_x E,
      \end{equation*}
      where $E$ is
      \begin{multline*}
        \{
          \left( x, A, B \right)
          \in C \times \mathcal{S}^d_+ \times \mathcal{S}^d_+
          \st \\
          g(x) = 0,\\
          p_x^{\prime}(0) = B_{00}, \\
          \frac{p_x^{(k+2)}(0)}{(k+2)!} =
          \sum_{\mathclap{\substack{i + j = k\\ 0 \leq i,j \leq d-1}}}
          A_{ij} - B_{ij}
          + \sum_{\mathclap{\substack{i + j = k + 1\\ 0 \leq i,j \leq d-1}}}
          B_{ij}
          \text{, for }  0 \leq k \leq 2d-2
        \}.
        \end{multline*}

    \item If the degree of $g$ is $2d+1$, then
      \begin{equation*}
        R_S(\xlp) = \proj_x E,
      \end{equation*}
      where $E$ is
      \begin{multline*}
        \{
          \left( x, A, B \right)
          \in C \times \mathcal{S}^{d+1}_+ \times \mathcal{S}^d_+
          \st \\
          g(x) = 0,\\
          p_x^{\prime}(0) = A_{00}, \\
          \frac{p_x^{\prime\prime}(0)}{2} = 2A_{01} + B_{00}, \\
          \frac{p_x^{(k+3)}(0)}{(k+3)!} =
          \sum_{\substack{i + j = k + 2\\ 0 \leq i,j \leq d}}
          A_{ij}
          + \sum_{\substack{i + j = k + 1\\ 0 \leq i,j \leq d -1}}
          B_{ij}
          - \sum_{\substack{i + j = k\\ 0 \leq i,j \leq d -1}}
          B_{ij}
          \text{, for }  0 \leq k \leq 2d-2
        \}.
        \end{multline*}
  \end{enumerate}
\end{theorem}
\begin{proof}
  We just prove the case of even degree as the proof for the odd degree case is
  similar.
  We have $x \in R_S(\xlp)$ if and only if $p_x(0) = 0$ and $p_x(\lambda)$
  is non-negative on $[0,1]$.
  By Theorem~\ref{thm:nonnegative}, this is equivalent to $p_x(0) = 0$ and there exist
  polynomials $s_1, s_2$ of degree $d$ and $d-1$, respectively,
  such that
  \[
    p_x(\lambda) = s_1(\lambda)^2 + \lambda(1 - \lambda) s_2(\lambda)^2.
  \]
  Given that $0 = p_x(0) = s_1(0)^2$, the polynomial $s_1$ has a root at $0$ and
  we can write it as $s_1(\lambda) = \lambda r_1(\lambda)$ where $r_1$ is
  a polynomial of degree $d-1$.
  Thus, $x \in R_S(\xlp)$ if and only if $p_x(0) = 0$ and there exist polynomials
  $r_1, r_2$ of degree $d-1$ such that
  \[
    p_x(\lambda) = \lambda^2 r_1(\lambda)^2 + \lambda(1 - \lambda) r_2(\lambda)^2.
  \]

  Let $\Lambda = (1, \lambda, \ldots, \lambda^{d-1})^\T$.
  The polynomials $r_i$ can be written as $r_i = c_i^\T \Lambda$ for some $c_i
  \in \mathbb{R}^{d}$.
  Then, $r_1(\lambda)^2 = \Lambda^\T A \Lambda$ and $r_2(\lambda)^2 = \Lambda^\T
  B \Lambda$ for some $A, B \in \mathcal{S}^d_+$.

  Thus, $x \in R_S(\xlp)$ if and only if $p_x(0) = 0$ and there exist
  $A, B \in \mathcal{S}^d_+$ such that
  \[
    p_x(\lambda) = \lambda^2 \Lambda^\T A \Lambda
    + \lambda(1 - \lambda) \Lambda^\T B \Lambda.
  \]
  Since $p_x(\lambda)$ is a polynomial of degree $2d$, its Taylor expansion at
  0 yields
  \[
    p_x(\lambda) = \sum_{k=1}^{2d} \frac{p_x^{(k)}(0)}{k!} \lambda^k.
  \]
  Identifying coefficients, we conclude the theorem.
\end{proof}
\begin{remark}
  One could also add the constraints $\rank(A) = \rank(B) = 1$ to $E$ in the
  statement of Theorem~\ref{thm:poly_vis}.
  The correctness can easily be seen from the proof since $A = c_1 c_1^\T$ and
  $B = c_2 c_2^\T$.
  Although it makes the set more restricted, the rank constraint is non-convex
  and does not change the projection.
  Thus, we decided to leave it out.
\end{remark}

We can recover Theorem~\ref{thm:qvis} from Theorem~\ref{thm:poly_vis}.
The set $E$ of Theorem~\ref{thm:poly_vis} for the quadratic case ($d = 1$) is described
by $g(x) = 0$, $p_x^{\prime}(0) = B_{00}$ and $p_x^{\prime\prime}(0)/2 = A_{00}
- B_{00}$, where $A_{00}, B_{00} \geq 0$.
This implies that $0 < g(\xlp) = p_x(1) = p_x^{\prime}(0)
+ p_x^{\prime\prime}(0)/2 = A_{00}$.
Therefore, $R_S(\xlp)$ consists of the $x$ such that $p_x(0) = 0$ and
$p^{\prime}_x(0) \geq 0$.
This last constraint is equivalent to $\langle \nabla g(x), \xlp -x \rangle \geq
0$ which is the only constraint needed, apart from $g(x) = 0$, to prove
Theorem~\ref{thm:qvis}.

The previous deduction is only possible because $V_S(\xlp) = R_S(\xlp)$ holds
for a quadratic constraint.
This equality does not hold as soon as the degree is greater than 2, even after
replacing $V_S(\xlp)$ by its closure, as shown in the following example.

\begin{example} \label{ex:closure}
  Consider $g(x_1, x_2) = (x_1^2 + x_2^2 - 1) x_1$, $S = \{ (x_1, x_2) \st
  g(x_1, x_2) \leq 0 \}$, and $\xlp = (1, -2)$.
  The set $S$ consists of the right half of the unit ball and the half space
  $x_1 \leq 0$ without the interior of the left half of the unit ball, see
  Figure~\ref{fig:counter}.
  The point $z = (-1,0)$ is not visible from $\xlp$, because $g(z + \lambda(\xlp
  - z)) = g(-1 + 2 \lambda, -2 \lambda) = ((2 \lambda - 1)^2 + 4 \lambda^2 - 1)
  (2 \lambda - 1) = 4 \lambda (2 \lambda - 1)^2 $ is zero at $\lambda
  = \frac{1}{2}$.
  On the other hand, $z \in R_S(\xlp)$ since $4 \lambda (2 \lambda - 1)^2 \geq
  0$ for every $\lambda \in [0,1]$.
  In this example $V_S(\xlp)$ is closed, so we conclude that $\cl V_S(\xlp)
  \neq R_S(\xlp)$.
  \begin{figure}
    \centering
    \includegraphics[width=.4\linewidth]{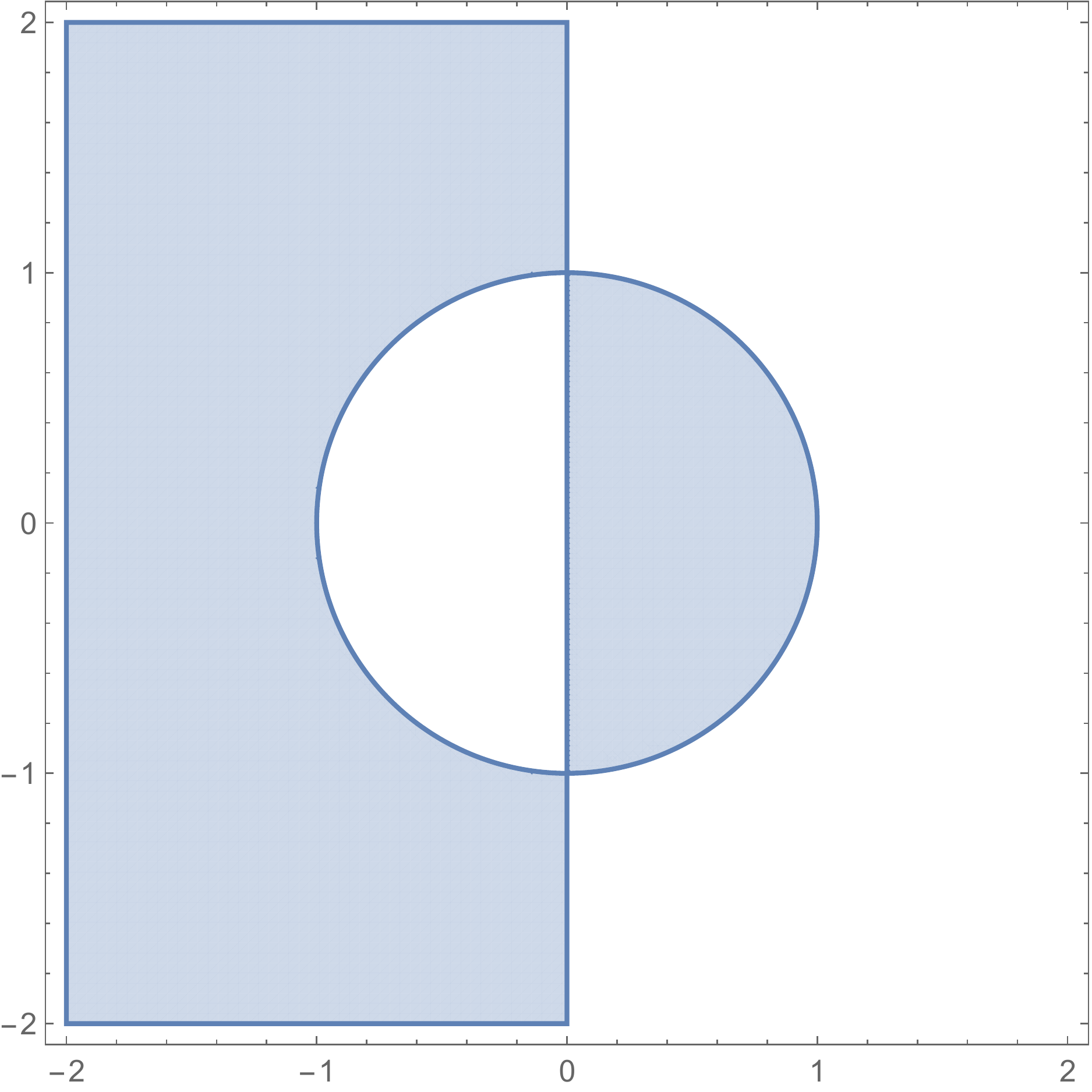}
    \caption{Feasible region $g(x) \leq 0$ of Example~\ref{ex:closure} that shows that
      $\cl V_S(\xlp) \neq R_S(\xlp)$ when the degree of $g$ is greater than 2.
    } \label{fig:counter}
  \end{figure}
\end{example}

\section{Conclusions and outlook}
Using the concept of visible points, we introduced a technique that allows to
reduce the domains in separation problems.
Such a result is particularly interesting for MINLP, since the tightness of the
domain directly affects the quality of underestimators, from which cuts are
obtained.

Some questions that could be interesting to look at in the future are the
following.
Is there a tighter domain other than $V_S$ that can be efficiently exploited?
Is there a useful characterizations of $V_S$ when $S$ contains more than one
non-convex constraint, in particular, if some variables are restricted to be
integer?

\paragraph{Acknowledgments}
The author is particularly indebted to Stefan Vigerske, Franziska Schl\"osser,
and Rainer Sinn for many helpful discussions.
He would also like to thank Juan Pablo Vielma, Benjamin M\"uller, and Daniel
Rehfeldt.

\bibliographystyle{spmpsci}
\bibliography{visible}

\end{document}